\documentclass{article}

\usepackage{amsmath}
\usepackage{amsthm}
\usepackage{amsfonts}
\usepackage{mathtools}
\usepackage[numbers]{natbib}
\usepackage{xcolor}
\usepackage{caption} 

\usepackage{hyperref}
\hypersetup{colorlinks=true, linkcolor=blue, citecolor=blue}

\newcommand{\R}{\mathbb{R}}
\newcommand{\C}{\mathbb{C}}

\newcommand{\cL}{\mathcal{L}}
\newcommand{\cA}{\mathcal{A}}

\newcommand{\cO}{\mathcal{O}}

\usepackage{bm}
\newcommand{\Ab}{\bm{A}}
\newcommand{\Bb}{\bm{B}}
\newcommand{\Eb}{\bm{E}}
\newcommand{\Mb}{\bm{M}}
\newcommand{\Pb}{\bm{P}}

\newcommand{\Pt}{\widetilde{P}}

\renewcommand{\exp}[1]{\mathrm{e}^{#1}}
\newcommand{\domain}[1]{\mathcal{D}(#1)}
\DeclareMathOperator{\Id}{Id}
\newcommand{\hfine}{\mathfrak{h}}
\renewcommand{\phi}{\varphi}

\DeclarePairedDelimiter{\iprod}{(}{)}

\DeclarePairedDelimiterX{\norm}[1]{\lVert}{\rVert}{#1}
\DeclarePairedDelimiter\normLH{\lVert}{\rVert_{\cL(H)}}
\DeclarePairedDelimiter\maxnorm{\lVert}{\rVert_{\infty}}
\DeclarePairedDelimiter\normLHh{\lVert}{\rVert_{\cL(H_h)}}

\newcommand{\diff}[1]{\mathrm{d}#1}
\newcommand{\ds}{\diff{s}}

\newcommand{\ddx}[1]{\frac{\mathrm{d}}{\diff{#1}}}
\newcommand{\ddjx}[2]{\frac{\mathrm{d}^{#1}}{\diff{#2}^{#1}}}
\newcommand{\dddx}[1]{\ddjx{2}{#1}}

\newcommand{\ddt}{\ddx{t}}
\newcommand{\dds}{\ddx{s}}
\newcommand{\ddds}{\dddx{s}}

\newcommand{\SH}{\Sigma(H)}
\newcommand{\SHP}{\Sigma^+(H)}
\newcommand{\SHh}{\Sigma(H_h)}
\newcommand{\SHhP}{\Sigma^+(H_h)}

\numberwithin{equation}{section}
\theoremstyle{plain}
\newtheorem{theorem}{Theorem}[section]
\newtheorem{lemma}[theorem]{Lemma}

\newtheorem{assumption}{Assumption}

\theoremstyle{definition}
\newtheorem{remark}{Remark}[section]

\title{Convergence analysis of a full discretization of operator-valued differential Riccati equations}

\author{Eskil Hansen \and Tony Stillfjord \and Teodor \r{A}berg}

\begin{document}

\maketitle

\begin{abstract}
  In recent previous work [E.~Hansen, T.~Stillfjord and T.~\r{A}berg, \emph{SIAM J.\ Numer.\ Anal.}, to appear], we analyzed the convergence of operator splitting methods applied to operator-valued differential Riccati equations (DRE). In this paper, we extend these results by analyzing the convergence of a full discretization based on finite elements in space and Lie splitting in time. As far as we are aware, this is the first such analysis for DRE. There are very few analyses of temporal discretizations of DRE overall, and none of them have been combined with spatial discretizations. However, it is clearly vital to know when the full discretization converges, since this is what will be used in practical applications.
  Our main result is that except for logarithmic factors, the method converges with order one in time and order two in space, under fairly weak assumptions on the problem data. This is illustrated by a numerical experiment based on an application in optimal control.
\end{abstract}

\section{Introduction} \label{sec:introduction}

This paper is a continuation of the work in~\cite{HansenStillfjordAaberg2026}, in which the convergence of splitting schemes applied to operator-valued differential Riccati equations (DRE) was studied. Here, we combine such a temporal discretization with a spatial discretization based on finite elements and analyze the convergence of the resulting full discretization.

A DRE is a semi-linear differential equation of the form
\begin{equation}\label{eq:DRE}
  \dot{P}(t) = A^* P(t) + P(t)A + Q - P(t)SP(t), \quad t \in (0,T], \quad P(0) = P_0.
\end{equation}
Such equations play a prominent role in optimal control, where their solutions provide the optimal feedback laws in finite-horizon linear quadratic regulator (LQR) problems. We consider the case when the problem data $A$, $Q$ and $S$ are all operators on a Hilbert space $H$, which means that also the solution values $P(t)$ are operators on $H$. This corresponds to optimal control of partial differential equations. To approximate these solutions in practice thus requires both a temporal and a spatial discretization.

Many different methods for DRE have been proposed in the literature. While essentially all spatial discretizations are based on finite elements, a wide variety of temporal discretizations have been suggested. We refer to~\cite{HansenStillfjordAaberg2026} for a recent overview. As also noted there, there are very few rigorous analyses of temporal discretizations of DRE in the literature. Spatial discretizations have been analysed more extensively, but mostly for algebraic Riccati equations. We are not aware of any paper which analyses a combination of temporal and spatial discretizations for DRE. 

The aim of this paper is therefore to provide a rigorous convergence analysis of such a combination of methods. As in~\cite{HansenStillfjordAaberg2026}, we consider the Lie splitting scheme in time, which iterates between approximating the solutions to the subproblems $\dot{P}(t) = A^* P(t) + P(t)A + Q$ and $\dot{P}(t) = - P(t)SP(t)$. This can be done much more efficiently than directly approximating the solution to~\eqref{eq:DRE}. In space, we consider a large class of discretizations that covers all standard finite element discretizations applicable to problems of the type $Ax = f$.

The analysis uses our temporal convergence results in~\cite{HansenStillfjordAaberg2026} and the spatial convergence results in~\cite{KrollerKunisch1991} as a starting point.
Let $h$ denote the spatial discretization parameter, let $\tau$ be the time step, and let $P_{n, h} \approx P(n\tau)$ denote the fully discretized approximation extended to an operator on $H$. Then our main result is that for all sufficiently small $h$ and $\tau$,
\begin{equation*}
  \normLH{P(n\tau) - \pi_h^* P_{n,h} \pi_h} \leq C 
    \bigl( h^2(n\tau)^{-1} + h^2(1 + \vert \log h \vert) 
      + \tau ( 1 +  \vert \log \tau \vert) \bigr)
\end{equation*}
That is, away from $t = 0$ the method is first-order convergent in time and second-order convergent in space, except for logarithmic factors which arise due to relatively weak assumptions on the problem. The singularity at $t = 0$ arises from a corresponding singularity of~\eqref{eq:DRE}. We emphasize that the convergence is measured in the operator norm $\normLH{\cdot}$ for bounded operators on a Hilbert space $H$. This is much stronger than the pointwise convergence of the type $\norm{P(n\tau)x - P_{n, h}x} \to 0$ for $x \in H$ that is frequently considered, e.g.\ in~\cite{BennerMena2016}.

In brief, the analysis approach is as follows:
\begin{enumerate}
\item We first compare the exact solution $P$ to the spatially semi-discretized approximation $P_h$, where the initial condition of the latter is a spatial discretization of $P_0$. The difference $\normLH{P(t) - P_h(t)}$ can be bounded using the results in~\cite{KrollerKunisch1991}.

\item Inspired by the approach in the lecture notes by Crouzeix~\cite{CrouzeixLecNotes}, we compare $P_h$ to $\Pt_h$ where $\Pt_h$ is the same spatial semi-discretization of~\eqref{eq:DRE} as $P_h$, but with a different, carefully chosen, regularized initial condition $\Pt_{0, h}$.

\item After proving that $\Pt_h$ is uniformly bounded with respect to $h$, we apply the results in~\cite{HansenStillfjordAaberg2026} to $\Pt_h(n\tau) - \Pt_{n,h}$, where $\Pt_{n, h}$ is the fully discretized DRE approximation with initial condition $\Pt_{0, h}$.

\item Finally, we bound $\Pt_{n, h} - P_{n, h}$ by following a similar procedure as to what would be used for the corresponding vector-valued elliptic problems.
 
\end{enumerate}

The rest of the paper is organized in the following manner. Section~\ref{sec:setting} specifies the details of the setting which we consider, and the main consequences of the assumptions on the problem data $A$, $S$, $Q$ and $P_0$. The properties of the family of spatial discretizations which fit into our framework are stated in Section~\ref{sec:spatial} together with important a priori bounds and convergence results. Then the details of the temporal discretization described in~\cite{HansenStillfjordAaberg2026} is reviewed in Section~\ref{sec:temporal}, and important auxiliary results are shown to be independent of $h$.
Section~\ref{sec:analysis_Lie} contains the main result and the details of the proof which is outlined in the previous paragraph. These theoretical results are illustrated by a numerical experiment in Section~\ref{sec:experiments}, and concluding remarks are given in Section~\ref{sec:conclusions}.

\section{Setting}\label{sec:setting}
The setting in which we will analyze the method is the same as in~\cite{HansenStillfjordAaberg2026}, but we restate the important notation here for the convenience of the reader.
Let $H$ be a Hilbert space over $\mathbb{C}$, and $U$, $Y$ be two auxiliary Hilbert spaces.
In the context of LQR problems, $U$ and $Y$ correspond to the input and output spaces, respectively, while $H$ is the state space.
We denote the inner product of a generic Hilbert space $X$ by $\iprod{\cdot,\cdot}_X$ and the induced norm by $\norm{\cdot}_X$. In both cases the subscript will be omitted if the space is clear from context.
Furthermore, $\cL(X, Z)$ denotes the Banach space of linear bounded operators from $X$ to another Hilbert space $Z$,
equipped with the usual operator norm. When $X = Z$, this space will be denoted by $\cL(X)$.
The Hilbert space adjoint of an operator $V: Z \to X$ is denoted by $V^*: X \to Z$.
The subspace of $\cL(X)$ consisting of self-adjoint operators, i.e.\ those satisfying
$\iprod{Vx,y} = \iprod{x,Vy}$ for all $x,y \in X$, will be denoted by $\Sigma(X)$. Those operators which are additionally non-negative, i.e.\ where $\iprod{Vx,x} \ge 0$ for all $x \in X$, will be denoted by $\Sigma^+(X)$.
Finally, we denote the usual supremum norm by
\begin{equation*}
  \maxnorm{P} = \sup_{t \in (0,T)} \norm{P(t)}_{H}.
\end{equation*}
Throughout the paper, $C$ will denote an arbitrary positive constant which
may be different from line to line. It will often depend on the problem data $T$, $\normLH{Q}$ and $\normLH{S}$, but for brevity we only specify its dependence on other parameters. 

We make the following standard assumptions in order to ensure the existence of solutions to~\eqref{eq:DRE} following the framework presented in~\cite{Bensoussan2007}.
\begin{assumption}\label{ass:operators}
  There are operators $B \in \cL(U,H)$ and $E \in \cL(H,Y)$ such that $Q$ and $S$ can be factorized as $Q=E^*E$ and $S=BB^*$. Further, $P_0 \in \SHP$, and the unbounded operator $A: \domain{A} \subset H \to H$ is the generator of an exponentially stable analytic $C_0$-semigroup $\exp{t A}$.
\end{assumption}
\begin{remark}
  The assumption of exponential stability is not a constraint, since if $A$ is the generator of an analytic semigroup that is not exponentially stable, then
  one can always find $\lambda$ such that the semigroup generated by $A-\lambda I$ is exponentially stable. While this shift alters~\eqref{eq:DRE}, the changes do not lead to substantive differences in the analysis.
To simplify the exposition, we therefore assume throughout that $\exp{t A}$ is exponentially stable, and outline the necessary modifications for when it is not in Appendix~\ref{app:reformulation}.
\end{remark}
Next, for $P \in \SH$ define the bilinear map $\tilde{\phi}_P(\cdot, \cdot) : \domain{A} \times \domain{A} \to \C$ by
\begin{equation*}
  \tilde{\phi}_P(x,y) = \iprod{Px,Ay} + \iprod{Ax,Py},
\end{equation*}
and consider the set of $P$ such that it is bounded with respect to the norm of $H$, that is
\begin{equation*}
  \domain{\cA} = \Big \{ P \in \SH : \vert \tilde{\phi}_P(x,y) \vert \leq C \norm{x}\norm{y} \; \text{for all } x,y \in \domain{A} \Big \}.
\end{equation*}
For such $P$, $\tilde{\phi}_P$ can be extended to an operator $\phi_P : H \times H \to \C$, and this gives rise to
the operator $\cA: \domain{\cA} \to \SH$ defined by
\begin{equation*}
  (\cA(P) x, y) = \phi_P(x,y) \quad \text{for all } x, y \in H.
\end{equation*}
This operator is the proper extension of the sum $A^* P + P A$, which is not necessarily in $\SH$. In particular, it holds that
\begin{equation*}
  A^* P x + P A x = \cA(P) x \quad \text{for all } x \in \domain{A}.
\end{equation*}
when $P \in \domain{\cA}$.
For further details and a proof of the previous statement, we refer to~\cite[Part IV]{Bensoussan2007}.
From now on, we will be working with the equation
\begin{equation}\label{eq:DRE_alt}
  \dot{P}(t) = \cA P(t) + Q - P(t) S P(t), \quad t \in (0,T], \quad P(0) = P_0,
\end{equation}
instead of~\eqref{eq:DRE}. Under Assumption~\ref{ass:operators}, this problem has a unique solution, see~\cite{HansenStillfjordAaberg2026} and~\cite[Part IV]{Bensoussan2007}. To guarantee sufficient regularity of this solution for the temporal discretization to converge with the expected order, we additionally make the following assumption, cf.~\cite{HansenStillfjordAaberg2026}:
\begin{assumption}
  \label{ass:p_0_reg}
  $P_0$ is in the domain of $\cA$.
\end{assumption}

Furthermore, there is a semigroup $\exp{t\cA}: \SH \to \SH$ associated with $\cA$, defined by
\begin{equation*}
  \exp{t\cA}P = \exp{t A^*} P \exp{t A}.
\end{equation*}
It is important to note that $\exp{t\cA}$ is not (in general) strongly continuous, which is why we do not say that it is \emph{generated} by $\cA$. Nevertheless, we have the pointwise continuity $\lim_{t \to 0} \exp{t\cA}Px = Px$ for any $x \in H$ and the important identity  $\ddt \exp{t\cA} P = \cA \exp{t\cA}P$ holds for $P \in \domain{\cA}$ when $t > 0$. The semigroup is also bounded in the way one would expect, as summarized in the following theorem:
\begin{theorem}\label{thm:sg_bounds}
  Let $A$ be the generator of an exponentially stable analytic semigroup. Then there exists a $\omega_0 < 0$ such that
  \begin{alignat*}{3}
    &\normLH{\exp{t A}} &&\leq C\exp{\omega_0 t}, &&\quad t \ge 0,\\
    &\normLH{A\exp{t A}} &&\leq \frac{C \exp{\omega_0 t}}{t}, &&\quad t > 0.
  \end{alignat*}
  Similarly, with $\cA$ defined as above and $M \in \cL(H)$,
  \begin{alignat*}{3}
    &\normLH{\exp{t\cA} M} &&\leq C\exp{2\omega_0 t} \normLH{M}, &&\quad t \ge 0, \\
    &\normLH{\cA\exp{t \cA} M} &&\leq \frac{C\exp{2 \omega_0 t}}{t} \normLH{M}, &&\quad t > 0.
  \end{alignat*}
\end{theorem}
\begin{proof}
  This follows directly from standard results on analytic semigroups in the same way as Lemma 2.5 in~\cite{HansenStillfjordAaberg2026}.
\end{proof}
\begin{remark}\label{remark:omega0}
  The bounds in Theorem~\ref{thm:sg_bounds} will be used in Lemma~\ref{lemma:cA_invertible} and~\ref{lemma:op_parabolic} with $\omega_0 < 0$ to ensure that $\cA$ and its approximations have suitable properties. In the rest of the paper, we typically use the slightly weaker bounds that arise from replacing $\omega_0$ by $0$.
\end{remark}

\subsection{Spatial discretization}\label{sec:spatial}
We assume that we have a family of finite dimensional subspaces $H_h \subset H$ depending on a discretization parameter $h > 0$, with $\norm{u}_{H_h} = \norm{\pi_h^*u}_{H}$ for any
$u$ in $H_h$. Here, the operator $\pi_h: H \to H_h$ denotes the $H$-orthogonal projection onto $H_h$ and $\pi_h^* : H_h \to H$ is the isometric injection operator.
We note that given an operator $P_h \in \cL(H_h)$, the natural extension to an operator in $\cL(H)$ is given by $\pi_h^* P_h \pi_h$, and it holds that
\begin{equation}\label{eq:op_extension}
  \normLH{\pi_h^* P_h \pi_h} = \normLHh{P_h}.
\end{equation}
Furthermore, we have a family of operators $A_h: H_h \to H_h$,
which serve as approximations of the operator $A$. 

Following~\cite[Chapter 4]{LasieckaTriggiani2000},we assume that the approximating spaces and operators 
fulfill the following bounds:
\begin{assumption}
  \label{ass:disc_updated}
  Let $A$ fulfill the assumptions given in Assumption~\ref{ass:operators} and let $h_0 > 0$. The following holds for all $h \le h_0$. First, the approximating spaces $H_h$ have the property that $\norm{x - \pi_h x} \to 0$ as $h \to 0$ for all $x \in H$.
  Furthermore,
  $A_h \in \cL(H_h)$ satisfies
  \begin{align}
    &\normLH{\pi_h^* \exp{t A_h} \pi_h} \leq C \exp{\omega_0 t},\quad t \ge 0,
    \label{eq:unif_analyt_0}\\
     &\normLH{\pi_h^* A_h \exp{t A_h} \pi_h} \leq C \frac{\exp{\omega_0 t}}{t}, \quad t > 0.
    \label{eq:unif_analyt_1}
  \end{align}
  Finally, the operator $A_h$ approximates $A$ in the sense that
  \begin{equation}
   \normLH{A^{-1} - \pi_h^* A_h^{-1}\pi_h } \leq C h^2.
   \label{eq:elliptic_bound}
  \end{equation}
\end{assumption}

\begin{remark}
  We note that $A^ {-1}$ exists as an operator in $\cL(H)$ by virtue of $A$ being the generator of an exponentially stable analytic semigroup. In particular, let $\rho < \omega_0$. Then $A + \rho I$ generates an analytic semigroup. By the characterization of analytic semigroups, see e.g.~\cite[Theorem 12.31]{RenardyRogers2004}, it follows that $(\lambda I - (A + \rho I))^{-1} \in \cL(H)$ for any $\lambda > 0$. Choosing $\lambda = \rho$ proves the assertion. That $A_h^ {-1}$ exists is part of Assumption~\ref{ass:disc_updated}.
\end{remark}

\begin{remark}
The inequalities~\eqref{eq:unif_analyt_0} and~\eqref{eq:unif_analyt_1} are analogous to Assumption A.1 in Section 4.1.2.1 of~\cite{LasieckaTriggiani2000}. We note that the $\omega_0$ in this assumption is not the same parameter as our $\omega_0$. The difference is that Assumption A.1 is formulated for possibly unstable $A$, while we assume that $A$ generates an exponentially stable semigroup. The former case can be converted into the latter by a shift $A \to A - \lambda I$ as described in Appendix~\ref{app:reformulation}. This is also used in~\cite{LasieckaTriggiani2000}, where the shifted operator is denoted $-\hat{A}$.
\end{remark}

\begin{remark}
  The restriction $h \le h_0$ will be enforced throughout the paper, but the choice of $h_0$ is arbitrary and it has a minor effect on the error constants. It could therefore, e.g., be chosen as the diameter of the computational domain.
\end{remark}

A main consequence of these assumptions on $A$ and $A_h$ is that we can also bound the corresponding solution operators in a similar way:

\begin{lemma}[{\cite[Proposition 4.1.2.1 (i)]{LasieckaTriggiani2000}}]\label{lemma:expAh_bound}
  Under Assumption~\ref{ass:operators} and Assumption~\ref{ass:disc_updated},
  \begin{equation*}
  \normLH{\pi_h^* \exp{t A_h}\pi_h - \exp{t A}} \leq C \exp{\omega_0t}\frac{h^2}{t}
\end{equation*}
for $t > 0$.
\end{lemma}

As in Remark~\ref{remark:omega0}, the bounds on the exponentials in~\eqref{eq:unif_analyt_0} will typically be used in the weaker form which arises from replacing $\omega_0$ by $0$, i.e., ignoring the decay when $t>0$.

We define $\cA_h : \cL(H_h) \to \cL(H_h)$ by
\begin{equation*}
  \cA_h U_h = A_h^* U_h + U_h A_h,
\end{equation*}
which induces the semigroup $\exp{t \cA_h}: \cL(H_h) \to \cL(H_h)$ given by
\begin{equation*}
  \exp{t \cA_h}U_h = \exp{t A^*_h} U_h \exp{t A_h}.
\end{equation*}
For these operators, the bounds of Assumption~\ref{ass:disc_updated} can be directly extended to yield the following analogue of Theorem~\ref{thm:sg_bounds}:
\begin{lemma}\label{lemma:unif_op_sg}
  Let Assumption~\ref{ass:operators} and~\ref{ass:disc_updated} hold. Then for any $M_h \in \cL(H_h)$,
  \begin{alignat*}{3}
    &\normLHh{\exp{t \cA_h} M_h} &&\leq C\exp{2\omega_0t} \normLHh{M_h}, &&\quad t \ge 0, \\
    &\normLHh{\cA_h \exp{t \cA_h} M_h} &&\leq \frac{C\exp{2\omega_0 t}}{t} \normLHh{M_h}, &&\quad t > 0.
  \end{alignat*}
\end{lemma}
\noindent
Furthermore, we approximate the operators $S$ and $Q$ by their projections $S_h \in \cL(H_h)$ and $Q_h \in \cL(H_h)$, defined by
\begin{equation*}
  S_h = \pi_h S \pi_h^*, \qquad Q_h = \pi_h Q \pi_h^*.
\end{equation*}
\begin{lemma}\label{lemma:ShQh_properties}
  Let Assumtion~\ref{ass:operators} hold, then the operators $S_h$ and $Q_h$ are both in $\SHhP$, and
  \begin{align*}
    \normLHh{S_h} &\le \normLH{S}, \\
    \normLHh{Q_h} &\le \normLH{Q}.
  \end{align*}  
\end{lemma}
\begin{proof}
The first assertion follows from the fact that
\begin{equation*}
  \iprod{\pi_h S \pi^*_h x, x}_H = \iprod{S \pi^*_h x, \pi^*_h x}_H \text{ for all } x \in H_h,
\end{equation*}
and similarly for $Q_h$. The second follows from $\norm{\pi_h}_{\cL(H, H_h)} = 1$.  
\end{proof}

We will then consider the (spatially) semi-discretized Riccati equation given by
\begin{equation}
  \label{eq:DREdisc_P}
  \begin{split}
    &\dot{P}_h(t) = \cA_h P_h(t) + Q_h - P_h(t) S_h P_h(t), \quad t \in (0,T], \\
    &P_h(0) = \pi_h P_0 \pi^*_h,
  \end{split}
\end{equation}
for $P_h(t) \in \SHhP$. 
For this problem there exist several convergence results. As far as the authors are aware, the following is the most sharp:
\begin{theorem}[{\cite[Theorem 3.1]{KrollerKunisch1991}}]
  \label{thm:KrollerKunish}
  Let Assumptions~\ref{ass:operators} and~\ref{ass:disc_updated} hold. Then the 
  spatial error is bounded by
  \begin{equation*}
    \normLH{P(t) - \pi_h^* P_h(t)\pi_h } \leq Ch^2 ( t^{-1} + \vert \log h \vert).
  \end{equation*}
\end{theorem}
\begin{remark}
  While the term $t^{-1}$ is unfortunate, it reflects the typical behaviour of the exact solution $P$ of~\eqref{eq:DRE_alt} near $t = 0$. It is unclear whether this term can be removed by imposing regularity constraints on $P_0$ (such as Assumption~\ref{ass:p_0_reg}), because the proof strategy in~\cite{KrollerKunisch1991} is based on the connection between Riccati equations and LQR problems rather than the direct approach considered here.
  We note that the singularity is avoided in~\cite{Cheung2025}, but at the price of working in the setting of compact linear operators. This severely restricts the applicability of the results.
  
  The term $\log h$ is removable under further assumptions on $P_0$, $U$ and the spaces $H_h$, see~\cite{Kroller1986}. But since its influence is marginal, we chose to include it in order to improve both the applicability of the results and the readability of the analysis.
\end{remark}
The approximation properties have been given in terms of $A$ and $A_h$, but we are 
interested in working with the operators $\cA$ and $\cA_h$ instead. The following
two results show that the bound given by~\eqref{eq:elliptic_bound} can be
extended to the operators $\cA$ and $\cA_h$. By the idea in~\cite[Theorem 4.1.3]{CurtainZwart2020}, we obtain the first result.
\begin{lemma} \label{lemma:cA_invertible}
  Let Assumptions~\ref{ass:operators} and~\ref{ass:disc_updated}. The operators $\cA : \domain{\cA} \to \SH$ and $\cA_h: \SHh \to \SHh$ are both invertible, independent of $h$.
  Furthermore, $\cA^{-1}: \SH \to \SH$ and $\cA_h^{-1}: \SHh \to \SHh$ are bounded.
\end{lemma}
\begin{proof}
  Consider first $\cA$. For an arbitrary $M \in \SH$ we define $P$ by
  \begin{equation*}
    Pz = \int_0^{\infty}{\exp{sA^*} M \exp{sA} z \ds}
  \end{equation*}
  for any $z \in H$.
  This is well-defined, since by Theorem~\ref{thm:sg_bounds} the norm of the integrand is bounded by $C^2\exp{2\omega_0 s} \norm{M} \norm{z}$. Since $\exp{sA^*} M \exp{sA}$ is also symmetric, $P \in \SH$.

  Now take $x, y \in \domain{A}$. We have
  \begin{align*}
    \iprod{Px, Ay} + \iprod{Ax, Py}
    &= \int_0^{\infty}{\iprod{\exp{sA^*} M \exp{sA} x, Ay} \,\ds} + \int_0^{\infty}{\iprod{Ax, \exp{sA^*} M \exp{sA} y} \,\ds}\\
    &= \int_0^{\infty}{\iprod{\exp{sA}x, M \exp{sA} Ay} \,\ds} +  \int_0^{\infty}{\iprod{\exp{sA}Ax, M \exp{sA} y} \,\ds}\\
    &= \int_0^{\infty}{\dds \iprod{\exp{sA}x, M \exp{sA} y} \,\ds}.
  \end{align*}
  Since the integrand satisfies
  \begin{equation*}
    \norm[\bigg]{\dds \iprod{\exp{sA}x, M \exp{sA} y}} \le 2C^2\exp{2\omega_0 s} \norm{M} \norm{Ax} \norm{Ay} \norm{x} \norm{y}
  \end{equation*}
  it is integrable on $[0, \infty)$, and we get
  \begin{equation*}
    \iprod{Px, Ay} + \iprod{Ax, Py} = \Bigl[ \iprod{\exp{sA}x, M \exp{sA} y} \Bigr]_{0}^{\infty}
                                    = - \iprod{x, My},
  \end{equation*}
  since $\exp{sA}z \to 0$ as $s \to \infty$ for any $z$. Hence
  \begin{equation*}
    \varphi_P(x, y) = - \iprod{Mx, y}
  \end{equation*}
  is continuous, so $P \in \domain{\cA}$ and $\cA P = -M$. Since the choice of $M$ was arbitrary, $\cA$ is surjective.
  To see that $\cA$ is also injective, take $P_1, P_2 \in \domain{\cA}$ and assume that $\cA(P_1) = \cA(P_2)$, i.e.
  \begin{equation*}
    \phi_{P_1}(x,y) - \phi_{P_2}(x,y) = 0 \quad \text{for all }x,y \in H.
  \end{equation*}
This implies that
  \begin{equation*}
    \iprod{Ax, (P_1 - P_2) y} + \iprod{(P_1 - P_2)x,  Ay} = 0 \quad \text{for all }x,y \in \domain{A}.
  \end{equation*}
  If $x_0, y_0 \in \domain{A}$ then both the particular choices $x = \exp{t A} x_0$ and $y = \exp{tA}y_0$ belong to $\domain{A}$ for any $t \ge 0$, so we find that
  \begin{equation*}
    \begin{split}
    & 0 = \iprod{A\exp{t A} x_0, (P_1 - P_2) \exp{tA}y_0} + \iprod{(P_1 - P_2)\exp{t A} x_0,  A\exp{tA}y_0} \\
    & \quad = \dds{\iprod{\exp{t A} x_0, (P_1 - P_2) \exp{tA}y_0}}.
    \end{split}
  \end{equation*}
  Thus $t \mapsto \iprod{\exp{t A} x_0, (P_1 - P_2) \exp{tA}y_0}$ is constant. But both $\exp{t A} x_0$ and $\exp{t A} y_0$ tend to $0$ as $t \to \infty$, since $\exp{tA}$ is exponentially stable. Thus the constant must be $0$.
  In particular with $t = 0$,
  \begin{equation*}
    0 = \iprod{x_0, (P_1 - P_2) y_0}, \quad \text{for all } x_0, y_0 \in \domain{A}.
  \end{equation*}
  Since $\domain{A}$ is dense in $H$, this means that $P_1 = P_2$, and thus $\cA$ is invertible.
  That $\cA^{-1}$ is bounded follows directly from
  \begin{equation*}\begin{split}
    \sup_{\normLH{M} = 1} \normLH{\cA^{-1}(M)} 
    &= \sup_{\normLH{M} = 1} \normLH[\bigg]{ -  \int_0^{\infty}{\exp{sA^*} M \exp{sA} \,\ds}} \\
    &\leq \sup_{\normLH{M} = 1}  \int_0^{\infty}{C^2 \exp{2\omega_0 s}\normLH{M} \,\ds}.
  \end{split}
  \end{equation*}

 By Assumption~\ref{ass:disc_updated}, specifically~\eqref{eq:unif_analyt_0}, the above arguments are also valid if $\cA$ and $H$ are replaced by $\cA_h$ and $H_h$.
\end{proof}

\begin{lemma}\label{lemma:op_parabolic}
  Let Assumptions~\ref{ass:operators} and~\ref{ass:disc_updated}. Then
  there exists a constant $C$, such that
  \begin{equation*}
    \normLH{\cA^{-1}(M) - \pi_h^* \cA_h^{-1} (\pi_h M \pi_h^*) \pi_h} \le C \normLH{M} h^2 (1 + |\log h|).
  \end{equation*}
\end{lemma}
\begin{proof}
  By the proof of Lemma~\ref{lemma:cA_invertible} we have that for any $z \in H$
  \begin{equation*}
    \cA^{-1}(M)z = -\int_0^{\infty}{\exp{sA^*} M \exp{sA}z \,\ds},
  \end{equation*}
  and similarly,
  \begin{equation*}
    \pi_h^* \cA_h^{-1}(\pi_h M \pi_h^*) \pi_h z = -\int_0^{\infty}{\pi_h^* \exp{sA_h^*} \pi_h M \pi_h^* \exp{sA_h} \pi_h z \,\ds}.
  \end{equation*}
  Thus
  \begin{align*}
    &\cA^{-1}(M)z - \pi_h^* \cA_h^{-1}(\pi_h M \pi_h^*) \pi_h z \\
    &= -\int_0^{\infty}{\Big(\exp{sA^*} - \pi_h^* \exp{sA_h^*} \pi_h\Big) M \exp{sA}z + \pi_h^* \exp{sA_h^*} \pi_h M \Big(\exp{sA}z - \pi_h^* \exp{sA_h} \pi_hz \Big) \,\ds} \\
    &=: \int_0^{\infty}{I(s)z \,\ds}.
  \end{align*}
  We split the integral into two parts, $I_1(z) = \int_0^{h^2}{I(s)z \,\ds}$ and $I_2(z) = \int_{h^2}^{\infty}{I(s)z \,\ds}$. By Theorem~\ref{thm:sg_bounds} and Assumption~\ref{ass:disc_updated} both exponentials are uniformly bounded, so $\norm{I_1(z)} \le Ch^2\normLH{M}\norm{z}$.
  By Lemma~\ref{lemma:expAh_bound}, we further get
    \begin{align*}
      \norm{I_2(z)} &\le \int_{h^2}^{\infty}{ C\exp{\omega_0s}\frac{h^2}{s} \norm{M \exp{sA}z} + \normLH{\pi_h^* \exp{sA_h^*} \pi_h M}  C\exp{\omega_0s}\frac{h^2}{s} \norm{z}\,\ds}  \\
                    &\le Ch^2\normLH{M}\norm{z}\int_{h^2}^{\infty}{ \exp{\omega_0s} \frac{1}{s} \,\ds} \\
                    &\le Ch^2\normLH{M}\norm{z} \Big ( \int_{h^2}^{1}{ \frac{1}{s} \,\ds}  
                      +\int_{1}^{\infty}{ \exp{\omega_0s} \,\ds} \Big )\\
    \end{align*}
    so that $\norm{I_2(z)} \le Ch^2(1 + |\log h|)\normLH{M}\norm{z}$. Combining the bounds on $I_1$ and $I_2$ and taking the supremum over $z$ with $\norm{z} = 1$ proves the result.
\end{proof}

\subsection{Full discretization}\label{sec:temporal}
In order to arrive at a full discretization of the DRE, we now discretize~\eqref{eq:DREdisc_P} in time using the Lie operator splitting scheme. We formulate it using the operators $F_h, G_h: \cL(H_h) \to \cL(H_h)$, which are defined by
\begin{equation*}
  \begin{split}
    &F_h M_h= \cA_h M_h + Q_h \quad \text{and} \\
    &G_h M_h = - M_h S_h M_h.
  \end{split}
\end{equation*}
These are analogous to the operators $F$ and $G$ in~\cite{HansenStillfjordAaberg2026} for the spatially non-discretized case.
They give rise to the subproblems
\begin{align*}
  \dot{M}_h &= F_h M_h , \quad M_h(0) = M_{0,h}, \\
  \dot{M}_h &= G_h M_h, \quad M_h(0) = M_{0,h},
\end{align*}
whose solutions we denote by
\begin{equation*}
  \exp{t F_h} M_{0,h} \quad \text{and} \quad \exp{t G_h} M_{0,h},
\end{equation*}
respectively, for any $M_{0,h} \in \Sigma^+(H_h)$.

The following auxiliary lemmas correspond to Lemma 2.11 and Lemma 2.12 in~\cite{HansenStillfjordAaberg2026} and follow directly from applying these results to~\eqref{eq:DREdisc_P} and using the uniform bounds on $S_h$ and $Q_h$ in Lemma~\ref{lemma:ShQh_properties}.
\begin{lemma}\label{lemma:unif_nlinflow}
  Let Assumption~\ref{ass:operators} hold and $M_{0,h} \in \Sigma^+(H_h)$. Then
  $\exp{tG_h} M_{0,h} = \bigl( I + t S_h M_{0,h} \bigr)^{-1} M_{0,h}$, defined on $[0,T]$, fulfills
  \begin{equation*}
    \maxnorm{t\mapsto\exp{tG_h} M_{0,h}} \leq \bigl( 1 + T \rho \normLH{S} \bigr)  \rho,
  \end{equation*}
  for all $M_{0,h}$ such that $\normLHh{M_{0,h}} \leq \rho$. Additionally,
  the derivatives of this function fulfill
  \begin{equation*}
    \maxnorm[\bigg]{\ddjx{j}{t} \bigg(t\mapsto\exp{t G_h} M_{0,h} \bigg) } \leq j! \normLH{S}^j \bigl( 1 + T \rho \normLH{S}\bigr)^{j+1} \rho^{j+1},
  \end{equation*}
  for $j= 1, \ldots$.
\end{lemma}
\begin{lemma}\label{lemma:unif_lipschitz}
  Let $M_{1,h}, M_{2,h} \in \cL(H_h)$ both fulfill $\normLHh{M_{i,h}} \leq \rho, i = 1,2$. Then 
  \begin{equation*}
    \normLHh{G_h M_{h,1} - G_h M_{h,2} } \leq 2\rho\normLH{S} \normLHh{M_{h,1} - M_{h,2}}.
  \end{equation*}
\end{lemma}

We consider a uniform temporal grid with step size $\tau$, such that $T = N\tau$. The fully discretized scheme is then defined by $P_{n, h} \approx P_h(n\tau)$ where $\pi_h^* P_h(n\tau) \pi_h \approx P(n\tau)$ and
\begin{align*}
  P_{n, h} = \cL_{h, \tau}^n P_{0, h}, \quad
  P_{0, h} = P_h(0),
\end{align*}
with
\begin{equation*}
  \cL_{h, \tau} = \exp{\tau F_h} \exp{\tau G_h}.
\end{equation*}
For a generic initial condition $M_{0,h} \in \SHhP$, the application of the operator $\cL_{h, \tau}$ can be formulated in the following way:
\begin{equation}\label{eq:scheme}
  \begin{split}
  \cL_{h, \tau} M_{0,h}
  &= \exp{\tau \cA_h} M_{0,h} + \int_0^\tau \exp{(\tau - s) \cA_h} Q_h \, \ds
    + \tau \exp{\tau \cA_h} G_h M_{0,h} \\
  &\quad + \int_0^\tau (\tau - s) \exp{\tau \cA_h} \ddds \exp{s G_h}M_{0,h} \, \ds \\
  &=: L_{0, h, \tau}(M_{0,h}) + L_{1, h,\tau} + L_{2, h, \tau}(M_{0,h})+ L_{3, h,\tau}(M_{0,h}).
  \end{split}
\end{equation}
This expansion will be used in the analysis in the next section.

\section{Full space-time convergence analysis}\label{sec:analysis_Lie}
Our main result, which will be proved in the rest of this section, is:
\begin{theorem}\label{theorem:main}
Let Assumption~\ref{ass:operators}, \ref{ass:p_0_reg} and~\ref{ass:disc_updated} be fulfilled. Then it holds that
\begin{equation*}
    \normLH{P(n\tau) - \pi_h^* P_{n,h} \pi_h} \leq C 
    \bigl( h^2(n\tau)^{-1} + h^2(1 + \vert \log h \vert) 
      + \tau ( 1 +  \vert \log \tau \vert) \bigr)
\end{equation*}
for sufficiently small $h$ and $\tau$. The error constant $C$ only depends on~$h_0$, $T$, $\normLH{P_0}$, $\normLH{\cA P_0}$, $\normLH{Q}$, and $\normLH{S}$.
\end{theorem}

Consider again the spatially semi-discretized DRE~\eqref{eq:DREdisc_P}.
To directly apply the temporal convergence results in~\cite{HansenStillfjordAaberg2026} to this equation, we would have to bound the derivative $\dot{P}_h$ over $(0, T]$, uniformly in $h$. This is unfortunately not straightforward, since $\normLHh{A_h} \to \infty$ as $h \to 0$ and $\dot{P}_h(t)$ approaches $\cA_h (\pi_h P_0 \pi_h^*)$ when $t \to 0$. For this reason, we introduce the regularized initial condition 
\begin{equation*}
\Pt_{0, h}:=\cA_h^{-1} ( \pi_h \cA P_0 \pi_h^*) 
\end{equation*}
together with the auxiliary equation
\begin{equation}
  \label{eq:DREdisc_Ptilde}
  \begin{split}
    &\dot{\Pt}_h(t) = \cA_h \Pt_h(t) + Q_h - \Pt_h(t) S_h \Pt_h(t), \quad t \in (0,T], \\
    &\Pt_h(0) =\Pt_{0,h},
  \end{split}
\end{equation}
where $\Pt_h(t) \in \SHhP$. Analogously to the full discretization $P_{n,h}$, we also define the auxiliary discretization $\Pt_{n, h}$ by
\begin{equation*}
  \Pt_{n, h} = \cL_{h, \tau}^n \Pt_{0, h}.
\end{equation*}
The initial value $\Pt_{0,h}$ is close to $P_{0,h}$, as shown in the next lemma. 
\begin{lemma}\label{lemma:smooth_init_bound}
Let Assumption~\ref{ass:operators}, \ref{ass:p_0_reg} and~\ref{ass:disc_updated} hold.
Then 
\begin{equation*}
  \normLH{ \pi_h^* P_{0,h} \pi_h - \pi_h^*\Pt_{0, h} \pi_h } \leq C \normLH{\cA P_0} h^2 ( 1 + \vert \log{h} \vert ) .
\end{equation*}
\end{lemma}
\begin{proof}
First note that both $P_{0,h}$ and $\Pt_{0,h}$ are elements of $\cL(H_h)$, and that $\pi_h \pi_h^*$ is the identity operator on $H_h$.
Thus,
\begin{equation*}
  P_{0,h} - \Pt_{0,h} = \pi_h \Bigl (P_0 -  \pi_h^* \cA_h^{-1} \bigl( \pi_h \cA P_0 \pi_h^*\bigr) \pi_h\Bigr )\pi_h^*.
\end{equation*}
Since $P_0 \in \domain{\cA}$ we may write $P_0 = \cA^{-1}(\cA P_0)$ and therefore
\begin{equation*}
  \begin{split}
    &\normLH{ \pi_h^* P_{0,h} \pi_h - \pi_h^*\Pt_{0, h} \pi_h }\\
    &\quad= \normLHh{ P_{0,h}  - \Pt_{0, h} } \\
    &\quad\leq\norm{\pi_h}_{\cL(H, H_h)} \normLH{\cA^{-1} (\cA P_0) -  \pi_h^* \cA_h^{-1}
      \bigl( \pi_h \cA P_0 \pi_h^*\bigr) \pi_h}\norm{\pi_h^*}_{\cL(H_h,H)} \\
    &\quad\leq C\normLH{\cA P_0}h^2 \bigl ( 1 + \vert \log h \vert \bigr),
  \end{split}
\end{equation*}
where we have used the identity~\eqref{eq:op_extension}, Lemma~\ref{lemma:op_parabolic}, with $ M = \cA P_0$,
and the fact that $\norm{\pi_h}_{\cL(H, H_h)} = \norm{\pi_h^*}_{\cL(H_h, H)} = 1$.
\end{proof}
We also have the following uniform bounds with respect to~$h$.
\begin{lemma}\label{lemma:unif_sol_norms}
  Let Assumption~\ref{ass:operators}, \ref{ass:p_0_reg} and~\ref{ass:disc_updated} hold.
  Then there exists a parameter~$\gamma_0>0$ such that 
  \begin{equation}\label{eq:unif_sol_norms}
\max \bigl\{ \maxnorm{\pi_h^* P_h \pi_h}, \maxnorm{\pi_h^* \Pt_h \pi_h}, \maxnorm{\pi_h^* \dot{\Pt}_h \pi_h }\bigr\} \leq \gamma_0
  \end{equation}
for all $h \leq h_0$. Here, $\gamma_0=\gamma_0(h_0,\normLH{P_0},\normLH{\cA P_0},T,\normLH{Q},\normLH{S})$.
\end{lemma}
\begin{proof}
By the same proof as in \cite[Lemmas IV.1.2.2--1.2.3, Theorem IV.1.2.1]{Bensoussan2007}, applied
to the finite dimensional problems~\eqref{eq:DREdisc_P} and~\eqref{eq:DREdisc_Ptilde}, respectively, we have that $\maxnorm{\pi_h^* P_h \pi_h},\maxnorm{\pi_h^* \Pt_h \pi_h}$ can be bounded 
in terms of $T,\normLHh{Q},\normLHh{S}$ and $\norm{P_{0,h}}_{\cL(H_h)},\norm{\Pt_{0, h}}_{\cL(H_h)}$. The terms $\normLHh{Q},\normLHh{S}$ are uniformly bounded with respect to $h$ by Lemma~\ref{lemma:ShQh_properties}, $\norm{P_{0,h}}_{\cL(H_h)}\leq \norm{P_0}_{\cL(H)}$, by construction, and Lemma~\ref{lemma:smooth_init_bound} gives
 \begin{equation*}
 \norm{\Pt_{0, h}}_{\cL(H_h)}=\normLH{\pi_h^*\Pt_{0, h} \pi_h} \leq \normLH{P_0} + Ch_0^2(1 + | \log h_0|) \normLH{\cA P_0}.
 \end{equation*}
Hence, the $P_h$ and $\Pt_h$ terms in~\eqref{eq:unif_sol_norms} are uniformly bounded with respect to $h$. 

Let $V_h = \dot{\Pt}_h$. As $\cA_h\Pt_{0,h}= \pi_h \cA P_0 \pi_h^*$ one has that
  \begin{equation*}
    \begin{aligned}
      \dot{V}_h(t) &= \cA_h V_h (t)  - V_h(t) S_h \Pt_h(t) - \Pt_h(t) S_h V_h(t), \quad t \in (0,T],\\
      V_h(0) &= \pi_h \cA P_0 \pi_h^*  + Q_h 
        - \Pt_{0,h} S_h \Pt_{0,h}.
    \end{aligned}
  \end{equation*}
By the variation of constants formula we can represent this as
  \begin{equation*}
    V_h(t) = \exp{t \cA_h}V_h(0) + \int_0^t \exp{(t-s)\cA_h} (\Pt_h(s) S_h V(s) + V(s) S_h \Pt(s)) \, \ds.
  \end{equation*}
Taking the norm results in
  \begin{equation*}
    \begin{split}
      \normLHh{V_h(t)&} \leq \normLHh{\exp{\cA_h t}}\normLHh{V_h(0)} \\
      &\quad + 2\int_0^t \normLHh{\exp{\cA_h (t-s)}}\normLH{S}\maxnorm{\pi_h^* \Pt_h \pi_h}\normLHh{V_h(t)}\,\ds \\
      &\leq C \normLHh{V_h(0)} + C\normLH{S}\maxnorm{\pi_h^* \Pt_h \pi_h} \int_0^t \normLHh{V_h(t)}\,\ds \\
      &\leq C \bigl(\normLH{\cA P_0}  + \normLH{S} \normLH{\pi_h^* \Pt_{0,h} \pi_h}^2 
        + \normLH{Q} \bigr) \\
      &\quad + C\normLH{S} \maxnorm{\pi_h^* \Pt_h \pi_h} \int_0^t \normLHh{V_h(t)}\,\ds.
    \end{split}
  \end{equation*}
The continuous Grönwall's inequality, see e.g.~\cite[Proposition 2.1]{Emmrich1999}, then yields that
 \begin{equation*}
\normLH{\pi_h^*\dot{\Pt}(t)\pi_h}=\normLHh{V_h(t)} \leq C(\normLH{\pi_h^* \Pt_{0,h} \pi_h})\exp{C\normLH{S}T\maxnorm{\pi_h^* \Pt_h \pi_h} }.
\end{equation*}
 As the right-hand-side does not depend on $t$ and the terms $\normLH{\pi_h^* \Pt_{0,h} \pi_h}$, $\maxnorm{\pi_h^* \Pt_h \pi_h}$ are uniformly bounded with respect to $h$, we have also derived a uniform bound for $\maxnorm{\pi_h^*\dot{\Pt}\pi_h}$.
\end{proof}
\begin{remark}
  The value $\gamma_0$ defined in Lemma~\ref{lemma:unif_sol_norms} will be used throughout the rest of the analysis without further reference to the lemma. 
\end{remark}

We can now rewrite the full error in the following way:
\begin{equation}\label{eq:error_formulation}
  \begin{split}
    \normLH{ P(n\tau)- \pi_h^*  P_{n,h} \pi_h}
    &\leq\normLH{P(n\tau) - \pi_h^* P_h(n\tau) \pi_h}\\
    &\quad+ \normLH{ \pi_h^* \bigl( P_h(n\tau) - \Pt_h(n\tau)\bigr) \pi_h }\\
    &\quad + \normLH{\pi_h^* \bigl(\Pt_h(n\tau) - \Pt_{n,h} \bigr) \pi_h }\\
    &\quad + \normLH{\pi_h^* \bigl(\Pt_{n,h} - P_{n,h} \bigr) \pi_h } \\
    &=: E_{1}^n + E_{2}^n + E_{3}^n + E_{4}^n.
  \end{split}
\end{equation}
This formalises the outline of the analysis given in Section~\ref{sec:introduction}; the $i$-th item in the list corresponds to $E_i^n$.
The first term $E_{1}^n$ is directly bounded in terms of~$h$ by Theorem~\ref{thm:KrollerKunish}. The following lemmas bound the remaining terms.
\begin{lemma}\label{lemma:init_cond_diff_bound}
Let Assumption~\ref{ass:operators}, \ref{ass:p_0_reg} and~\ref{ass:disc_updated} be fulfilled. Then
\begin{equation*}
E_{2}^n \leq C(\gamma_0,\normLH{\cA P_0})\,h^{2}( 1 + | \log h | )
\end{equation*}
for all $n = 0, \ldots, N$, and $h \le h_0$.
\end{lemma}
\begin{proof}
The operator $W_h = P_h - \Pt_h$ fulfills the equation
\begin{equation*}
  \begin{split}
      &\dot{W}_h = \cA_hW_h - P_h S_h P_h + \Pt_h S_h \Pt_h = \cA_hW_h - W_h S_h P_h - \Pt_h S_h W_h, t\in(0,T],\\
      &W_h(0) = P_{0,h} - \Pt_{0,h},
  \end{split}
\end{equation*}
 and we have the representation
\begin{equation*}
  W_h(t)
    = \exp{t \cA_h} \bigl( P_{0,h} - \Pt_{0,h} \bigr) 
    -\int_0^t \exp{(t-s) \cA_h} \bigl( W_h(s) S_h P_h(s) - \Pt_h(s) S_h W_h(s) \bigr)\ds.
\end{equation*}
By Lemmas~\ref{lemma:unif_op_sg}, \ref{lemma:smooth_init_bound} and~\ref{lemma:unif_sol_norms} we have the bound
\begin{equation*}
    \norm{W_h(t)}_{\cL(H_h)}
    \leq C \normLH{\cA P_0} h^2 ( 1 + \vert \log{h} \vert ) + C(\gamma_0) \int_0^t \norm{W_h(s)}_{\cL(H_h)} \ds.
\end{equation*}
The continuous Grönwall's inequality once again implies that
\begin{equation*}
    \normLH{ \pi_h^* \bigl( P_h(t) - \Pt_h(t)\bigr) \pi_h}=\norm{W_h(t)}_{\cL(H_h)}
    \leq C \normLH{\cA P_0} h^2 ( 1 + \vert \log{h} \vert )  \exp{C(\gamma_0) T}
 \end{equation*}
for all $t\in [0,T]$, and the desired bound follows.
\end{proof}
\begin{lemma}\label{lemma:infdim_to_unif}
Let Assumption~\ref{ass:operators}, \ref{ass:p_0_reg}, and~\ref{ass:disc_updated} hold. Then there exist a constant $\tau_0>0$, such that 
\begin{equation*} 
\normLH{\pi_h^*\Pt_{n,h}\pi_h}\le \gamma_0+1\quad\text{and}\quad E_{3}^n \leq C(\gamma_0)\tau ( 1 + \vert \log \tau \vert )
\end{equation*}
for all $n = 0, \ldots, N$,  $h \le h_0$, and $\tau\leq \tau_0$. Here, $\tau_0$ depends on the parameter~$\gamma_0$ and the problem data $T,\normLH{Q},\normLH{S}$.
\end{lemma}
\begin{proof}
For a fixed $h$, the result follows from~\cite[Theorem 3.6]{HansenStillfjordAaberg2026} applied to~\eqref{eq:DREdisc_Ptilde}. What needs to be verified is that there  are $h$-independent constants $\tau_0$ and $C(\gamma_0)$. The crucial element of the proof of~\cite[Theorem 3.6]{HansenStillfjordAaberg2026} is to choose a parameter $\gamma$ such that
 \begin{equation}\label{eq:gamma}
 \max \bigl\{\maxnorm{\pi_h^* \Pt_h \pi_h}, \maxnorm{\pi_h^* \dot{\Pt}_h \pi_h }\bigr\} \leq \gamma.
 \end{equation}
Lemma~\ref{lemma:unif_sol_norms} enables the choice $\gamma = \gamma_0$, which implies that the bound~\eqref{eq:gamma} holds for all $h \leq h_0 $. The constants also depend on $\normLHh{S_h}$, $\normLHh{Q_h}$, $\normLHh{\exp{\tau \cA_h}}$, and $\normLHh{\cA_h \exp{\tau \cA_h}}$. The first two of these terms are uniformly bounded by Lemma~\ref{lemma:ShQh_properties}, and the last two by Lemma~\ref{lemma:unif_op_sg}. Thus no $h$-dependent bounds are required in the derivation of the constants.
\end{proof}
It remains to bound $E^n_{4}$. This requires an additional induction argument, as $\cL_{\tau, h}^n$ is not globally Lipschitz continuous.
\begin{lemma}\label{lemma:bounded_expanse_induction}
Let Assumption~\ref{ass:operators}, \ref{ass:p_0_reg} and~\ref{ass:disc_updated} hold. If $\normLH{\pi_h^*P_{k,h}\pi_h} \leq \gamma_0+2$ for $k = 0, \ldots, n$, then
\begin{equation*}
E^{n+1}_{4} \leq D_0 \bigl( h^{2} (1 + \vert \log h \vert ) + \tau \bigr)
\end{equation*}
for all $h \le h_0$ and $\tau\le \tau_0$. Here, $D_0$ only depends on the parameter~$\gamma_0$ and the problem data $T,\normLH{Q},\normLH{S}$.
\end{lemma}
\begin{proof}
Using~\eqref{eq:scheme} to recursively express the $(n+1)$th error in terms of the errors at previous time steps gives
\begin{equation*}
    \begin{split}
      &P_{n+1,h}- \Pt_{n+1,h} \\
      &= \exp{n\tau \cA_h} \bigl( P_{0,h} - \Pt_{0,h} \bigr)  + \sum_{k=0}^{n} \exp{ (n-k) \tau \cA_h} \tau  ( G_h P_{k,h} - G_h \Pt_{k,h} \bigr)  \\
      & \quad + \sum_{k=0}^{n} \exp{ (n-k) \tau \cA_h} \int_0^\tau (\tau - s) \bigl(\ddds  \exp{s G_h} P_{k,h}- \ddds \exp{s G_h} \Pt_{k,h}\bigr) \, \ds .
    \end{split}
\end{equation*}
Taking the norm and using  Lemmas~\ref{lemma:unif_op_sg}, \ref{lemma:unif_nlinflow}, \ref{lemma:unif_lipschitz}, \ref{lemma:smooth_init_bound}, and \ref{lemma:infdim_to_unif} together with the hypothesis $\normLH{\pi_h^*P_{k,h}\pi_h}\le \gamma_0+2$ results in
\begin{equation*}
    E^{n+1}_{4} \leq C\normLH{\cA P_0} h^2 ( 1 + \vert \log h \vert)+C(\gamma_0) \tau \sum_{k=0}^n E^k_{4} + C(\gamma_0) \tau.
\end{equation*}
 Thus, the desired bound follows by applying a discrete Grönwall inequality, see e.g.\ \cite[Proposition 4.1]{Emmrich1999}.
\end{proof}
\begin{lemma}\label{lemma:method_bounded_expanse}
  Let Assumption~\ref{ass:operators}, \ref{ass:p_0_reg} and~\ref{ass:disc_updated}. Then there exist constants $h^* \leq h_0$ and $\tau^*\leq \tau_0$ such that 
\begin{equation*}
E^n_{4} \leq D \bigl(h^{2} ( 1 + \vert \log h \vert ) + \tau \bigr),
\end{equation*}
for all $n = 0, \ldots, N $, $h \leq h^*$, and $\tau \leq \tau^*$. Here, $D$ only depends on the parameter~$\gamma_0$ and the problem data $T,\normLH{Q},\normLH{S}$.
\end{lemma}
\begin{proof}
By Lemma~\ref{lemma:smooth_init_bound}, we can choose $h^* \leq h_0$ and $\tau^*\leq \tau_0$ such that 
\begin{equation*}
\normLH{\pi_h^*\Pt_{0,h}\pi_h-P_{0}}\leq D\bigl(h^{2} ( 1 + \vert \log h \vert ) + \tau \bigr)\leq 1
\end{equation*}
for all  $h\le h^*$ and $\tau\le \tau^*$, where $D\geq D_0$. Let $h\le h^*$ and $\tau\le \tau^*$ and assume that $E^k_{4}\le D\bigl(h^{2} ( 1 + \vert \log h \vert ) + \tau \bigr)$ for $k=0,\ldots, n$. Lemma~\ref{lemma:infdim_to_unif} then gives
\begin{equation*}
\normLH{\pi_h^*P_{k,h}\pi_h}\leq \normLH{\pi_h^*\Pt_{k,h}\pi_h}+E^k_{4}\le (\gamma_0+1)+1,
\end{equation*}
and Lemma~\ref{lemma:bounded_expanse_induction} implies that $E^{n+1}_{4}\le D_0\bigl(h^{2} ( 1 + \vert \log h \vert ) + \tau \bigr)$. As $E^0_{4}=\normLH{\pi_h^*\Pt_{0,h}\pi_h-P_{0}}$, the sought after bound follows by induction. 
\end{proof}

Combining Theorem~\ref{thm:KrollerKunish}, Lemmas~\ref{lemma:init_cond_diff_bound}, \ref{lemma:infdim_to_unif},  and~\ref{lemma:method_bounded_expanse} to bound each of the four terms in~\eqref{eq:error_formulation} now proves the main result, Theorem~\ref{theorem:main}.

\section{Numerical experiments}\label{sec:experiments}
We illustrate the theoretical results by applying the full discretization to a DRE arising from a LQR problem. We use the same setup as in~\cite{HansenStillfjordAaberg2026}, but generalized from 1D to 2D. Thus 
$H = L^2(\Omega)$ with $\Omega = (0,1) \times (0,1)$. We choose $A$ to be the Laplacian with periodic boundary conditions, i.e.\ $A = \Delta$ and $\domain{A} = H^2_{\mathbb{T}}(\Omega)$. With $Ex = \int_{(1/2, 1) \times (0,1)} x$, the problem is a controlled heat equation on the unit square, where we observe the mean of the state $x$ on one half of the domain $\Omega$.
We let the control consist of a single input signal which scales a function on $\Omega$. Thus $B : \mathbb{R} \rightarrow H$ with $Bu = u \xi$ where we let $\xi \in H^2(\Omega)$. Then $B^*v = \iprod{v, \xi}_H$ so that $S = BB^* : H \rightarrow H$ satisfies $Sv = BB^*v = \iprod{\xi,v}_H \xi$. We construct the initial condition $P_0 = G^*G: H \rightarrow H$ in a similar way, by setting $G^*x = x \zeta$, where $\zeta \in H^2(\Omega)$.
By construction, Assumption~\ref{ass:operators} is now satisfied. Since $G$ maps into $H^2(\Omega)$, so is Assumption~\ref{ass:p_0_reg}.

For the spatial discretization, we use the finite element method with piecewise linear basis functions $\phi_j$, $j=1, \ldots, N_h$. Thus $H_h := \operatorname{span}\big(\{\phi_j\}_{j=1}^{N_h}\big)$, and the operators $A_h$ are defined via $\iprod{A_h u_h , v_h}_{H_h} = \iprod{A u_h, v_h}_{H} = -\iprod{\nabla u_h, \nabla v_h}_H$ for all $u_h, v_h \in H_h$. We use a simple grid where $\Omega$ is first cut into equi-sized squares which are then subdivided into triangles. With $N_x$ squares in each cardinal direction, we have $\cO(N_x^2)$ degrees of freedom in total.

Since the grid is quasi-uniform, we have the bound $\norm{\pi_h x - x} \le Ch^2\|Ax\|$, see e.g.~\cite[Inequality (5.39)]{LarssonThomee2009}. Combined with the dense embedding of $\domain{A}$ in $H$, this shows that $\norm{\pi_h x - x} \to 0$ as $h \to 0$ for any $x \in H$. Inequality~\eqref{eq:elliptic_bound} follows from e.g.~\cite[Theorem 5.24 and (5.25)]{LarssonThomee2009}. Further, the inequalities~\eqref{eq:unif_analyt_0} and~\eqref{eq:unif_analyt_1} can be found in~\cite[Proposition 4.1.2.1]{LasieckaTriggiani2000}.
In conclusion, also Assumption~\ref{ass:disc_updated} is fulfilled.

As outlined in~\cite{HansenStillfjordAaberg2026}, the approximation $P_h(t)$ can be represented by a matrix $\Pb(t) \in \R^{N_h \times N_h}$ which solves a matrix-valued DRE
\begin{equation*}
\Mb \dot{\Pb} \Mb = \Mb \Pb \Ab + \Ab^T \Pb \Mb + \Eb^T \Eb - \Mb \Pb \Bb \Bb^T \Pb \Mb,
\end{equation*}
where $\Mb$ is the mass matrix, $\Ab$ is the stiffness matrix and $\Eb$ and $\Bb$ are the matrix representations of $E$ and $B$, respectively. We get the matrix representation of $P_{n,h}$ by applying the Lie splitting scheme to this equation. Here, we choose $N_x = 2^k$, $k = 2, \ldots, 7$, corresponding to approximations $\Pb(t) \in \R^{N_x^2 \times N_x^2}$ and to $h = 2^{-k}$. For each $N_x$, we try different numbers of time steps, $N_t = 2^j$, $j = 1, \ldots, 13$, corresponding to $\tau = 2^{-j}$.

Since an exact solution is not available and cannot be ``manufactured'' without destroying the structure of the DRE, we use the approximation with $N_x = 2^7$ and $N_t = 2^{13}$ as a reference solution $\{P^{\text{ref}}_{n}\}_{n=1}^{2^{13}}$. When comparing it to $\{P_{n,h}\}_{n=1}^{2^j}$, we restrict it to $\{P^{\text{ref}}_{2^{j-13}n}\}_{n=1}^{2^j}$ so that the approximations are given on the same temporal grid with $2^j$ points. In order for the spatial grids also agree, we extend $P_{n,h}$ to $\Id_h^{\hfine} P_{n,h} (\Id_h^{\hfine})^*$, where $\hfine = 2^{-8}$ and $\Id_h^{\hfine}$ is the injection operator from $H_h$ to $H_{\hfine}$. 
In total, we then compute the relative discretized $\maxnorm{\cdot}$-errors
\begin{equation*}
  \textnormal{err}_{\tau, h} = \frac{\max_{n = 1, \ldots, 2^j} \normLH{\Id_{\hfine} (\Id_h^{\hfine} P_{n,h} (\Id_h^{\hfine})^* - P^{\text{ref}}_{2^{j-13}n})(\Id_{\hfine})^*}}{\max_{n = 1, \ldots, 2^j} \normLH{\Id_{\hfine} P^{\text{ref}}_{2^{j-13}n} (\Id_{\hfine})^*}}
\end{equation*}
when $\tau = 2^{-j}$.

From Theorem~\ref{theorem:main}, we expect that $\textnormal{err}_{\tau, h}$ behaves as $\cO(h^2 + \tau)$  except for logarithmic factors. This is indeed what we observe in Figure~\ref{fig:error_total}.
\begin{figure}
  \centering
  \includegraphics[width=\textwidth]{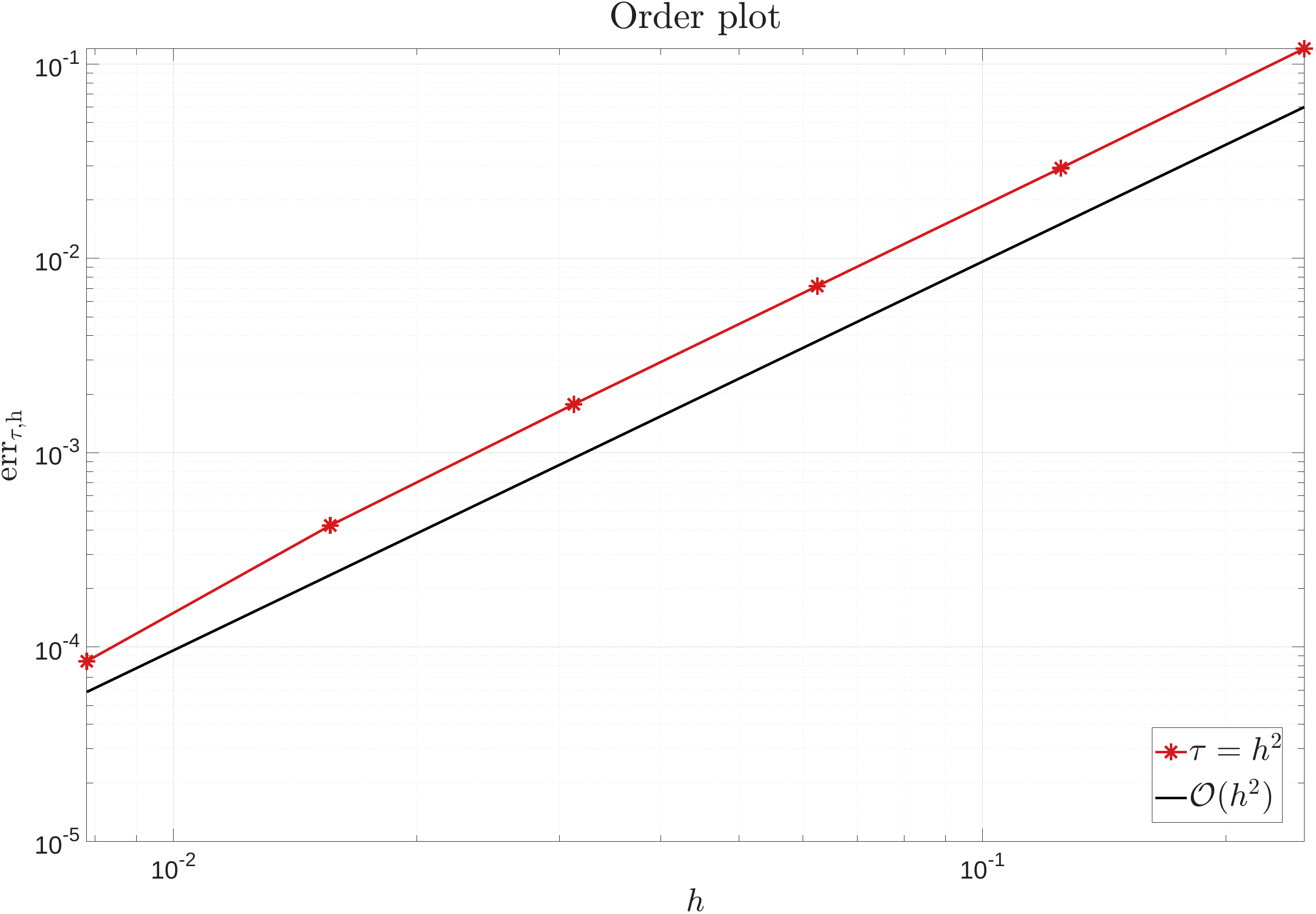}
  \caption{The errors $\textnormal{err}_{\tau, h}$ plotted versus $h$ when $\tau = h^2$. The errors converge like $\cO(h^2)$ as expected.}
  \label{fig:error_total}
\end{figure}
In Figure~\ref{fig:error_tau}, we have additionally plotted the errors against $\tau$, with one curve for each $h$. Here, we can see how the error decreases like $\cO(\tau)$ when the temporal errors are dominant, but eventually stagnates at the level of the spatial error. Conversely, Figure~\ref{fig:error_h} shows how the error decreases like $\cO(h^2)$ until stagnating due to the temporal error.
\begin{figure}
  \centering
  \includegraphics[width=\textwidth]{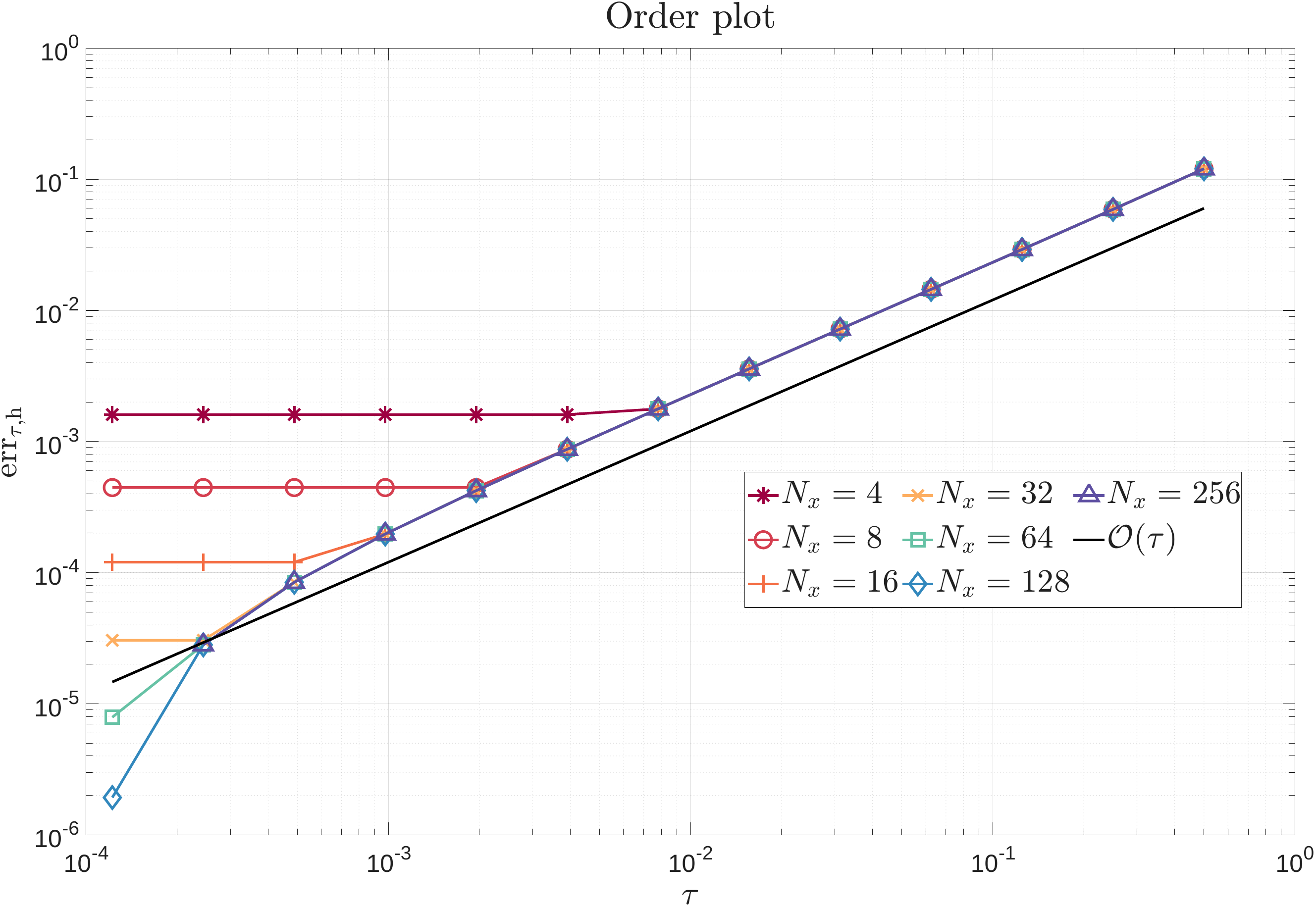}
  \caption{The errors $\textnormal{err}_{\tau, h}$ plotted versus $\tau$ for different $h$. The errors converge like $\cO(\tau)$ until stagnating due to the spatial error.}
  \label{fig:error_tau}
\end{figure}

\begin{figure}
  \centering
  \includegraphics[width=\textwidth]{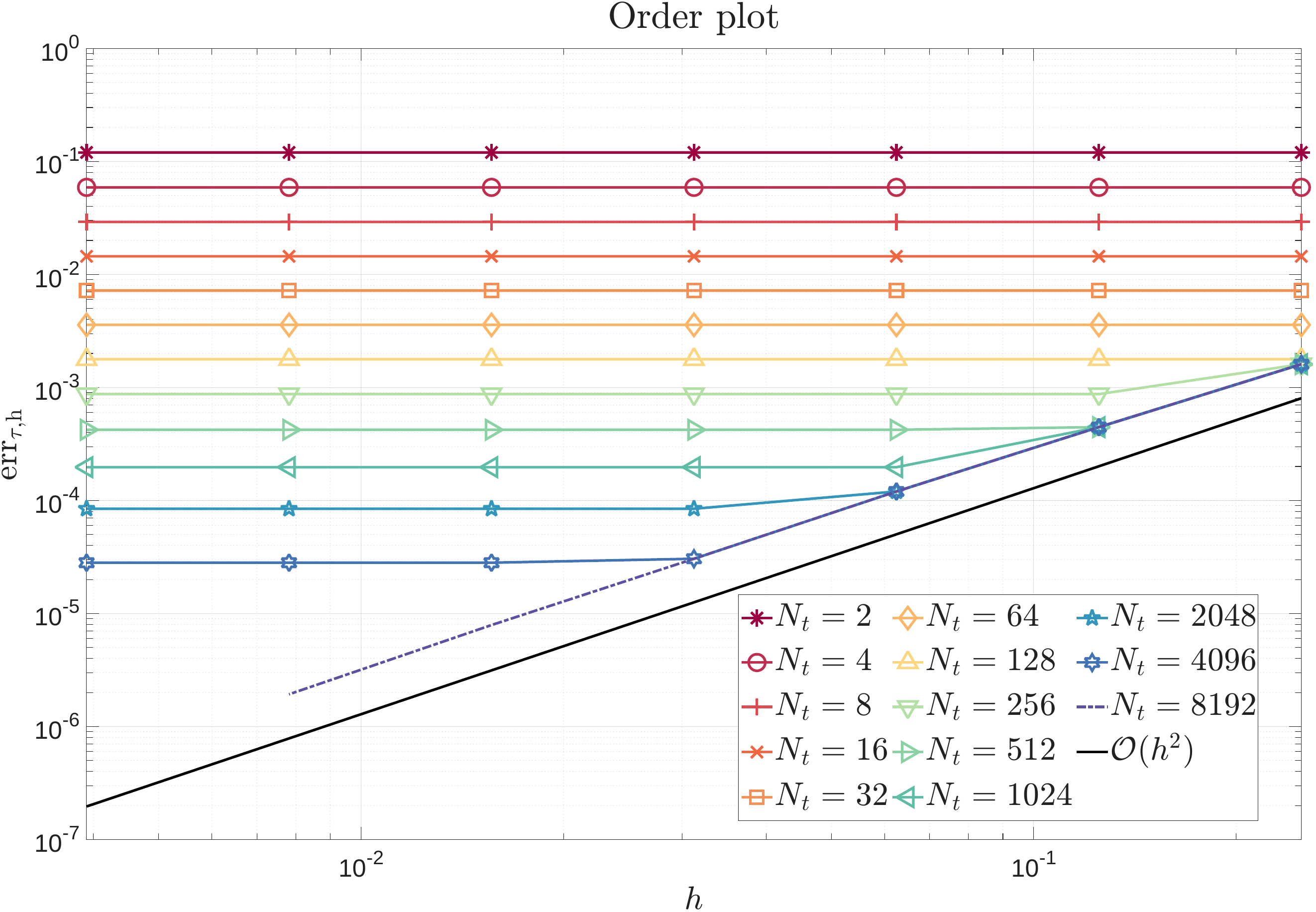}
  \caption{The errors $\textnormal{err}_{\tau, h}$ plotted versus $h$ for different $\tau$. The errors converge like $\cO(h^2)$ until stagnating due to the temporal error.}
  \label{fig:error_h}
\end{figure}

\section{Conclusions}\label{sec:conclusions}
We have proved convergence of a full discretization of operator-valued DRE with optimal orders, except for logarithmic factors. This behaviour is also observed in practice when the stated assumptions are fulfilled. These are fairly mild, in the sense that they cover all DREs arising from distributed control LQR applications, as long as the initial condition $P_0$ is sufficiently regular. Compared to~\cite{HansenStillfjordAaberg2026}, only the assumption on the spatial discretization is new, and this covers all standard FEM discretizations. We note that the analysis does not cover DREs arising from boundary control applications. These result in either an unbounded operator $B$ or an unbounded $C$, or both. In either case, the DRE solution suffers a loss of regularity, which has a direct effect on the speed of convergence that any discretization can achieve. Simply stating such problems properly requires a much more technical framework and as far as we are aware there are no spatial convergence results similar to~\cite{KrollerKunisch1991}. A significant amount of preliminary work is therefore required before a rigorous analysis of a full discretization such as the one presented here can be attempted. It is our aim to consider this in the future.

\section*{Acknowledgments}
The authors were partially supported by the Swedish Research Council under grants 2023-03982 (EH, TS, TÅ) and 2023-04862 (EH, TS).

\bibliography{refs}
\bibliographystyle{abbrvnat}

\appendix
\section{The case when $A$ does not generate an exponentially stable semigroup}\label{app:reformulation}

In this section, we assume that $A$ satisfies all the statements in Assumption~\ref{ass:operators}, except that the generated semigroup $\exp{tA}$ is not exponentially stable. In this case, the analyticity of $\exp{tA}$ implies that there exists a $\lambda > 0$ such that $A - \lambda I$ generates an exponentially stable semigroup. We can therefore make a change of variables in~\eqref{eq:DRE} such that the transformed equation has a $A$ which satisfies all of Assumption~\ref{ass:operators}. In the following, we describe how this affects the DRE and the convergence analysis.

Let $P$ be the solution to~\eqref{eq:DRE} and define
\begin{equation*}
  \bar{P}(t) = \exp{-2 \lambda t} P(t).
\end{equation*}
Then $\dot{\bar{P}}(t) = -2\lambda \bar{P}(t) + \exp{-2\lambda t} \dot{P}(t)$, i.e.
\begin{equation*}
  \dot{\bar{P}}(t) = \bar{A}^* \bar{P}(t) + \bar{P}(t) \bar{A} + \bar{Q}(t) - \bar{P}(t) \bar{S}(t) \bar{P}(t),
\end{equation*}
where
\begin{align*}
  \bar{A} &= A - \lambda I, \\
  \bar{Q}(t) &= \exp{-2\lambda t} Q, \\
  \bar{S}(t) &= \exp{2\lambda t} S.
\end{align*}
Now $\bar{A}$ generates an exponentially stable semigroup, and the solution $\bar{P}$ has the same regularity and positivity properties as $P$ since the transformation consists of multiplication by a scalar-valued positive function. For the same reason, the time-varying terms $\bar{Q}$ and $\bar{S}$ do not cause issues. In particular,
\begin{align*}
  \maxnorm{\bar{Q}} &\le \normLH{Q},\\
  \maxnorm[\bigg]{\ddjx{j}{t}\bar{S}(t)} &\le (2\lambda)^j\exp{2\lambda T} \normLH{S}, \quad j = 0, 1, \ldots.
\end{align*}
By using these quantities instead of $\normLH{Q}$ and $\normLH{S}$ where necessary, the results in~\cite{HansenStillfjordAaberg2026} and the present paper carry over with minor modifications.
Some notable differences include:
\begin{itemize}
\item The variation of constants formula is now
  \begin{equation*}
    \dot{\bar{P}}(t) = \exp{t\bar{\cA}} P_0 + \int_0^t \exp{(t-s)\bar{\cA}} \bar{Q}(s) \ds,
  \end{equation*}
  which gives rise to the $\maxnorm{\bar{Q}}$-terms.

\item The time-varying nonlinearity $G(t) \in \cL(\SH)$ defined by $G(t)P = -P\bar{S}(t)P$ is locally Lipschitz-continuous uniformly in $t$: If $\normLH{P_j} \le \rho$, $j = 1, 2$, then
  \begin{equation*}
    \normLH{G(t)P_1 - G(t)P_2} \le 2\rho \maxnorm{\bar{S}} \normLH{P_1 - P_2}.
  \end{equation*}
  Thus Lemma 2.12 in~\cite{HansenStillfjordAaberg2026} and  Lemma~\ref{lemma:unif_lipschitz} in the present paper still hold.
  
\item The solution to the nonlinear sub-problem $\dot{\bar{P}}(t) = - \bar{P}(t) \bar{S}(t) \bar{P}(t)$, $\bar{P}(0) = P_0$, is given by
  \begin{align*}
    \bar{P}(t) &= \bigg( I + P_0 \int_0^t \bar{S}(\tau) \diff{\tau}  \bigg)^{-1} P_0 \\
               &= \bigg( I + \frac{\exp{2\lambda t} - 1}{2\lambda} P_0 S \bigg)^{-1} P_0.
  \end{align*}
  Thus Lemma 2.11 in~\cite{HansenStillfjordAaberg2026} and Lemma~\ref{lemma:unif_nlinflow} in the present paper still hold.
  
  \item In~\cite{HansenStillfjordAaberg2026}, the $S$ in the decomposition of the Lie-splitting rest-term $R$ into integrals of $\lambda_1$ and $\lambda_2$ in Lemma 3.3 and Lemma 3.10, needs to be evaluated at $t+r$. In addition, $\lambda_2$ contains the extra term
    \begin{equation*}
      -\bar{P}(t+r) \dot{\bar{S}}(t+r) \bar{P}(t+r).
    \end{equation*}
    which becomes the extra term $-\bar{P}(s) \dot{\bar{S}}(s) \bar{P}(s)$  in $f$. But due to the uniform bound on $\bar{S}$, this only marginally changes the error constants.
\end{itemize}

\end{document}